\documentclass[10pt]{article}
\usepackage{bbm}
\usepackage{amsmath,amsthm,amssymb}
\usepackage{tabularx}
\usepackage{enumitem}
\usepackage{mathrsfs}
\usepackage{caption}
\usepackage{graphicx} 
\usepackage{authblk}
\usepackage[hypertexnames=false,colorlinks=true,urlcolor=blue,linkcolor=blue, citecolor=blue]{hyperref}
\usepackage[numbers,comma,square,sort&compress]{natbib}
\usepackage[a4paper,text={6.5in,10in},centering]{geometry}

\graphicspath{{Figures/}}

\makeatletter
\newtheorem*{rep@theorem}{\rep@title}
\newcommand{\newreptheorem}[2]{%
\newenvironment{rep#1}[1]{%
 \def\rep@title{#2 \ref*{##1}}%
 \begin{rep@theorem}}%
 {\end{rep@theorem}}}
\makeatother

\newtheorem{thm}{Theorem}
\newtheorem{lem}[thm]{Lemma}

\newtheorem*{hyp*}{Hypotheses}

\newtheorem{rmk}[thm]{Remark}

\newreptheorem{lem}{Lemma}
\newreptheorem{thm}{Theorem}

 {\begin{trivlist}\item[]\textbf{Acknowledgments.}}{\end{trivlist}}

 {\begin{trivlist}\item[]\textbf{Data availability statement.}}{\end{trivlist}}

\begin{document}

\title{On the interlacing property for zeros of Bessel functions}
\author{Dan J. Hill}
\affil{\small Mathematical Institute, University of Oxford, Oxford, OX2 6GG, UK}

\maketitle
\begin{abstract}

\noindent This note presents a simple approach to proving the interlacing properties of positive zeros of Bessel functions of the first kind. The approach relies only on the standard recurrence relations between Bessel functions and characterising intersections between curves, providing a more accessible and intuitive understanding for the interlacing behaviour of the zeros of Bessel functions.


\end{abstract}

	\vspace{4pt}
	
	
	


\section{Introduction}

Bessel functions of the first kind, denoted by $J_{k}(r)$ for some $k\in\mathbb{R}$ and $r\geq0$, are a class of special functions defined as solutions of Bessel's differential equation
\begin{equation*}
    \left(\frac{\mathrm{d}^2}{\mathrm{d}r^2} + \frac{1}{r}\frac{\mathrm{d}}{\mathrm{d}r} + 1 - \frac{k^2}{r^2}\right)u(r) = 0.
\end{equation*}
which remain bounded as $r\to0$. Bessel functions arise as particular solutions of the Helmholtz equation, where the index $k$ corresponds to the angular dependence of the solution, and often appear in orthonormal bases and bifurcation problems in cylinders (such as in Erhardt, Wahl\'en \& Weber~\cite{Erhardt} and the recent work by Wheeler~\cite{Wheeler}). In these cases, it is often important to characterise the positive zeros $j_{k,n}$ of the Bessel functions $J_{k}(r)$ and the order in which they occur for increasing values of $r$. Here, the index $k$ of the zero $j_{k,n}$ coincides with the index of the Bessel function, while the index $n\in\mathbb{N}$ denotes an ordering (i.e. $0<j_{k,1} < j_{k,2} < \dots$).

There have been many mathematical studies regarding the zeros of Bessel functions, where we direct the reader to Chapter XV in Watson~\cite{Watson} for a historical overview (we also recommend reviews by Elbert~\cite{Elbert} and Laforgia \& Natalini~\cite{Laforgia} for more recent results). Early works by Gegenbauer~\cite{Gegenbauer} and Porter~\cite{Porter} in the 1890's provided simple proofs for the interlacing of zeros for neighbouring (in $k$) Bessel functions,
\begin{equation*}
    j_{k,n} < j_{k+1, n} < j_{k,n+1} < j_{k+1,n+1}, \qquad\qquad \forall k>-1, \quad n\in\mathbb{N},
\end{equation*}
using recurrence relations of the Bessel functions and simple tools from calculus. This simple interlacing structure becomes more complex as one considers the zeros of further apart (in $k$) Bessel functions, and the mathematical proofs likewise increase in complexity. 

There has been an increase of interest in studying the interlacing problem for Bessel functions (see for example Cho \& Chung~\cite{Cho}, Chung, Lee \& Park~\cite{Chung}, and P\'almai \& Apagyi~\cite{Palmai}); these approaches mostly use tools from ordinary differential equations, such as Sturm's comparison principle, or employ other concepts like Lommel polynomials and integral formulae for Bessel functions. The aim in this note is not to compete with the generality presented in these other works, but rather to present a simple approach to prove the interlacing properties of the zeros $j_{k,n}$ that does not require any sophisticated analytic tools or additional concepts beyond basic properties of Bessel functions. In particular, our entire approach will rely on three simple identities for Bessel functions $J_k(r)$: the following two identities
\begin{equation}\label{id:Bessel-1}
    J_{k-1}(r) + J_{k+1}(r) = \frac{2k}{r}\,J_{k}(r),\qquad \qquad
    J_{k-1}(r) - J_{k+1}(r) = 2\, \frac{\mathrm{d}}{\mathrm{d}r}J_{k}(r)
\end{equation}
which hold for all $k\in\mathbb{R}$, $r>0$, as well as the Tur\'an inequality
\begin{equation}\label{id:Turan}
    J_{k+1}(r)\,J_{k-1}(r) - J_{k}(r)^2 < \frac{1}{k+1}\,J_k(r)^2
\end{equation}
which holds for $k>0$, $r>0$; a proof of \eqref{id:Turan} was first given by Sz\'asz~\cite{Szasz} in 1950. 

Considering the simplicity of this approach, along with the additional insight that it provides towards the interlacing zeros problem, it is remarkable that no one has previously presented these ideas. And yet, as far as this author is aware, this approach appears to be novel in this context. We hope that this work can help provide an accessible approach to characterise the zeros of Bessel functions, thereby demystifying a phenomenon that appears in many physical applications.

\section{Characterisation of zeros}
The key insight of this approach is to consider the function
\begin{equation}\label{def:Fk}
    F_{k}(r) := \frac{r\,J_{k}(r)}{J_{k+1}(r)}
\end{equation}
for $k>-1$ and $r>0$, for which it is straightforward to prove the following properties.
\begin{lem}\label{lem:F1}
For each fixed $k>-1$, the graph of $F_{k}(r)$ consists of countably many disjoint curves $\{\mathcal{B}_n\}_{n\in\mathbb{N}}$ formed as the graphs of functions $\phi_n(r): (j_{k+1,n-1}, j_{k+1,n})\to\mathbb{R}$ for $n\in\mathbb{N}$, respectively, where we denote $j_{k+1,0}=0$. Furthermore, each $\phi_n(r)$ is a smooth, strictly decreasing function on its domain, with $\phi_1(r) < 2\,(k+1)$ for all $r\in(0,j_{k+1,1})$ and
\begin{equation*}
 \phi_n(j_{k,n}) = 0, \qquad\qquad \phi_{n}(r) \to -\infty, \quad r \nearrow j_{k+1,n}, \qquad\qquad \phi_{n+1}(r) \to \infty, \quad r\searrow j_{k+1,n},
\end{equation*}
for all $n\in\mathbb{N}$.
\end{lem}
\begin{proof}
We first note, using \eqref{id:Bessel-1}, that 
\begin{equation*}
F_{k}'(r) = \frac{r \,(J_{k}(r)J_{k+2}(r) - J_{k+1}(r)^2)}{J^2_{k+1}(r)},
\end{equation*}
and so, by the Tur\'an inequality~\eqref{id:Turan},
\begin{equation*}
F_{k}'(r) < -\frac{r}{k+2}
\end{equation*}
for all $k>-1$\footnote{Note we are applying \eqref{id:Turan} for the index $k+1$ rather than $k$, which is why the region of validity is $k>-1$.}. Hence, $F_{k}(r)$ is a strictly decreasing function at every point where it is differentiable. The graph of $F_{k}(r)$ has vertical asymptotes wherever $J_{k+1}(r) =0$ and roots when $J_{k}(r) = 0$; i.e., the vertical asymptotes and zeros of the graph of $F_{k}(r)$ lie at the points $r = j_{k+1,n}$ and $r=j_{k,n}$ for $n\in\mathbb{N}$. The bound $F_{k}(r) < 2(k+1)$ for $r\in(0,j_{k+1,1})$ follows from $F_k(r)$ being a strictly decreasing function, along with asymptotic expansions for $J_n$ at $r=0$.
\end{proof}

\begin{figure}[ht!]
    \centering
    \includegraphics[width=0.8\linewidth]{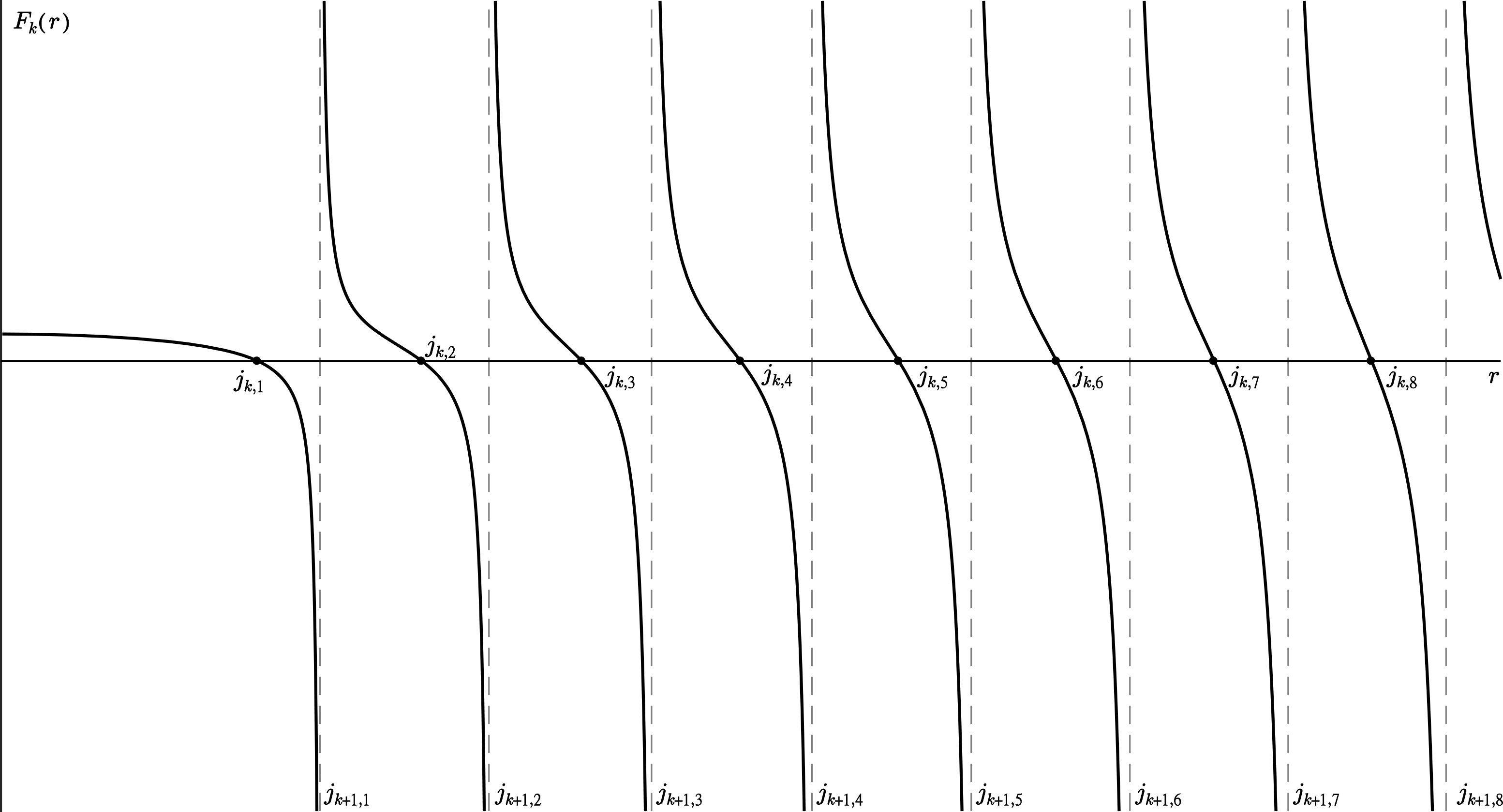}
    \caption{Schematic plot of $F_k(r)$ for any $k>-1$, as described in Lemma~\ref{lem:F1}}
    \label{fig:Bessel}
\end{figure}

An important consequence of Lemma~\ref{lem:F1} is that the qualitative shape of $F_k$ remains the same as $k$ varies, which allows us to consider a general $k>-1$ in our subsequent analysis; see Figure~\ref{fig:Bessel} for a visualisation of $F_k(r)$. We now express functions of the form $r^\ell\,J_{k+1+\ell}(r)/ J_{k+1}(r)$ in terms of $F_k(r)$, which we present in the following lemma.
\begin{lem}\label{lem:recurr}
    For each $k>-1$, the function $a_\ell := r^\ell\,J_{k+1+\ell}(r)/ J_{k+1}(r)$ is a solution to the recurrence relation
\begin{equation}\label{e:recurr}
a_{\ell+2} =2(k+2+\ell)a_{\ell+1} - r^2\,a_\ell, \qquad\qquad a_0 = 1, \qquad a_1 = 2(k+1) - y
\end{equation}
    for all $\ell\in\mathbb{N}_0$, $r>0$, where we have defined $y = F_k(r)$.
\end{lem}
\begin{proof}
We recall that the Bessel function $J_k$ satisfies the recurrence relation \eqref{id:Bessel-1}, and so
\begin{equation*}
\begin{split}
F_k(r) ={}& 2(k+1) - \frac{r\,J_{k+2}(r)}{J_{k+1}(r)}.\\
\end{split}
\end{equation*}
Sending $k\mapsto k+1+\ell$ and then multiplying both sides by $r^{\ell+1}\,J_{k+2+\ell}(r)/J_{k+1}(r)$, we obtain
\begin{equation*}
\begin{split}
r^2\,\frac{r^\ell\,J_{k+1+\ell}(r)}{J_{k+1}(r)} ={}& 2(k+2+\ell)\frac{r^{\ell+1}\,J_{k+2+\ell}(r)}{J_{k+1}(r)} - \frac{r^{2+\ell}\,J_{k+3+\ell}(r)}{J_{k+1}(r)},
\end{split}
\end{equation*}
for any $\ell\in\mathbb{N}_0$. Recalling the definitions of $a_\ell$ and $y$, we note that these two equations correspond to 
\begin{equation*}
    a_1 = 2(k+1) - y, \qquad\qquad r^2 a_{\ell} = 2(k+2+\ell) a_{\ell+1} - a_{\ell+2}
\end{equation*}
with $a_0 = 1$ by definition, and thus $r^{\ell}\,J_{k+1+\ell}(r)/J_{k+1}(r)$ satisfies \eqref{e:recurr}.
\end{proof}

By deriving an implicit relationship between $F_k(r)$ and functions of the form $r^{n+1}\,J_{k+2+n}(r)/J_{k+1}(r)$ (which can be made explicit through iterative calculations), we are now able to associate the zeros $j_{k+\ell,n}$ with intersections between the graph of $F_k$ and other curves, which we present in the following theorem.

\begin{thm}\label{thm:intersect}
    Fix $k>-1$, $\ell\in\mathbb{N}$ and let $a_\ell = a_\ell(y,r)$ be a solution of \eqref{e:recurr}. The zeros $j_{k+1+\ell, n}$ correspond to intersections between the graphs of $F_k$ and $G_{k,1+\ell}$, where $G_{k,1+\ell}(r)$ is found by solving $a_{\ell}(G_{k,1+\ell}, r)=0$. In particular,
    \begin{equation*}
\begin{split}
G_{k,2}(r) ={}& 2(k+1)\\
G_{k,3}(r) ={}& 2(k+1) - \frac{r^2}{2(k+2)}\\
G_{k,4}(r) ={}& 2(k+1) - \frac{2(k+3)\,r^2}{\left(4(k+2)(k+3) - r^2\right)}\\
G_{k,5}(r) ={}& 2(k+1) - \frac{\left(4(k+3)(k+4) - r^2\right)r^2}{4(k+3)\left(2(k+2)(k+4) - r^2\right)}\\
G_{k,6}(r) ={}& 2(k+1)  - \frac{4(k+4)\left(2(k+3)(k+5) - r^2\right)r^2}{\left(16\,(k+5)(k+4)(k+3)(k+2) - 12\,(k+3)(k+4)r^2 + r^4\right)}\\
G_{k,7}(r) ={}& 2(k+1) - \frac{\left(16(k+6)(k+5)(k+4)(k+3) - 12(k+5)(k+4)r^2 + r^4\right)r^2}{2(k+4)\left(16(k+6)(k+5)(k+3)(k+2) - 16(k+5)(k+3)r^2 + 3 r^4\right)}\\
\end{split}
\end{equation*}
\end{thm}
\begin{proof}
By Lemma~\ref{lem:recurr}, we obtain
\begin{equation*}
    \frac{r^\ell J_{k+1+\ell}(r)}{J_{k+1}(r)} = a_{\ell}(F_k(r),r),
\end{equation*}
and so, evaluating at $r = j_{k+1+\ell,n}$, it follows that 
\begin{equation*}
    a_{\ell}(F_k(j_{k+1+\ell,n}),j_{k+1+\ell,n}) = 0.
\end{equation*}
By the definition of $G_{k,1+\ell}(r)$, it follows that $F_k(j_{k+1+\ell,n}) = G_{k,1+\ell}(j_{k+1+\ell,n})$. Likewise, $F_k(r_*) = G_{k,1+\ell}(r_*)$ for some $r_*>0$ implies that $a_{\ell}(F_k(r_*),r_*) = \frac{r_*^\ell J_{k+1+\ell}(r_*)}{J_{k+1}(r_*)} =  0$, and so $r_* = j_{k+1+\ell,m}$ for some $m\in\mathbb{N}$. We solve \eqref{e:recurr} iteratively for low orders of $n$, where we obtain
\begin{equation*}
\begin{split}
a_1 ={}& \left[2(k+1) - y\right],\\
a_2 ={}& 2(k+2)\left[2(k+1) - y\right] - r^2,\\
a_3 ={}& \left(4(k+2)(k+3) - r^2\right)\left[2(k+1) - y\right] - 2(k+3)\,r^2,\\
a_4 ={}& 4(k+3)\left(2(k+2)(k+4) - r^2\right)\left[2(k+1) - y\right] - \left(4(k+3)(k+4) - r^2\right)r^2,\\
a_5 ={}& \left(16\,(k+5)(k+4)(k+3)(k+2) - 12\,(k+3)(k+4)r^2 + r^4\right)\left[2(k+1) - y\right] \\
&\qquad - 4(k+4)\left(2(k+3)(k+5) - r^2\right)r^2,\\
a_6 ={}& 2(k+4)\left(16(k+6)(k+5)(k+3)(k+2) - 16(k+5)(k+3)r^2 + 3 r^4\right)\left[2(k+1) - y\right] \\
&\qquad - \left(16(k+6)(k+5)(k+4)(k+3) - 12(k+5)(k+4)r^2 + r^4\right)r^2.\\
\end{split}
\end{equation*}    
Setting $a_{\ell}=0$ and substituting $y = G_{k,1+\ell}(r)$ yields the expressions for $G_{k,2}, \dots, G_{k,7}$.
\end{proof}

\begin{rmk}
    We note that, if $a_\ell, a_{\ell+1}$ have the form
\begin{equation*}
    a_{j} = P_{j}(r^2)\,\left[2(k+1) - y\right] - Q_{j}(r^2)\,r^2,
\end{equation*}
    where $P_{j}, Q_{j}$ are polynomials of degree $n_j,m_j$ for each $j=\ell,\ell+1$, respectively, then it follows that
\begin{equation*}
    a_{\ell+2} = P_{\ell+2}(r^2)\,\left[2(k+1) - y\right] - Q_{\ell+2}(r^2)\,r^2,
\end{equation*}
    where $P_{\ell+2}, Q_{\ell+2}$ are polynomials of respective degree $\max\{n_\ell+1, n_{\ell+1}\}$ and $\max\{m_\ell+1,m_{\ell+1}\}$. It follows from the initial forms of $a_1$ and $a_2$ that $a_{\ell}$ can be written as
\begin{equation*}
    a_{\ell} = P_{\ell}(r^2)\,\left[2(k+1) - y\right] - Q_{\ell}(r^2)\,r^2,
\end{equation*}
for all $\ell\in\mathbb{N}\backslash\{1\}$, where $P_{\ell}, Q_{\ell}$ are polynomials of degree $\lfloor\frac{\ell-1}{2}\rfloor$ and $\lfloor\frac{\ell-2}{2}\rfloor$, respectively. Consequently, the graph of $G_{k,\ell+1}$, given by the explicit form
\begin{equation*}
    G_{k,\ell+1}(r) = 2(k+1) - \frac{Q_{\ell}(r^2)\,r^2}{P_{\ell}(r^2)},
\end{equation*}
will have at most $\lfloor\frac{\ell-1}{2}\rfloor$ asymptotes, $|G_{k,\ell+1}(r)|\to\infty$ as $r\to\infty$ if $\ell$ is even, and $|G_{k,\ell+1}(r)|\to C_{k,\ell}$ with $C_{k,\ell}$ constant if $\ell$ is odd.
\end{rmk}

It remains to characterise the ordering of each zero $j_{k+\ell,n}$ based on the intersections of $F_k$ with $G_{k,\ell}$. We note that, since $F_k$ is a strictly decreasing function over each region $(j_{k+1,n}, j_{k+1,n+1})$, it is sufficient to show that $G_{k,\ell_1}(r) > G_{k,\ell_2}(r)$ for all $r\in(j_{k+1,n}, j_{k+1,n+1})$ in order to conclude that any zeros of $J_{k+\ell_1}(r)$ in $(j_{k+1,n}, j_{k+1,n+1})$ must occur before any zeros of $J_{k+\ell_2}(r)$.

In the following theorem, we consider the interlacing properties of zeros of $\{J_{k}, J_{k+1}, J_{k+2}, J_{k+3}, J_{k+4}\}$; one can of course extend these results further, but the conditions for different interlacing properties become more complex.
\begin{thm}\label{thm:interlacing}
    Fix $k>-1$, $n\in\mathbb{N}$ and define $r_k:=2\sqrt{(k+1)(k+2)}$, $\hat{r}_k:=\sqrt{2(k+1)(k+3)}$; the zeros $j_{k+2,n}, j_{k+3,n}$ lie in the interval $(j_{k+1,n}, j_{k+1,n+1})$ and satisfy the following interlacing properties:
    \begin{equation*}
        \begin{cases}
            j_{k+1,n} < j_{k+2,n} < j_{k+3,n} < j_{k,n+1} < j_{k+1,n+1}, & \quad \text{if $j_{k,n+1} < r_k$,}\\
            j_{k+1,n} < j_{k+2,n} < j_{k,n+1} < j_{k+3,n} <  j_{k+1,n+1}, & \quad \text{if $j_{k,n+1} > r_k$.} \\
        \end{cases}
    \end{equation*}
    Furthermore, the zero $j_{k+4,n}$ lies in the interval
    \begin{equation*}
        j_{k+4,n}\in\begin{cases}
            (j_{k+3,n}, j_{k,n+1}), &  \quad \text{if $j_{k+1,n+1}<r_{k+1}$ and $j_{k,n+1} < \hat{r}_{k}$,}\\
         (j_{k,n+1}, j_{k+1, n+1}), & \quad \text{if $j_{k+1,n+1}<r_{k+1}$ and $j_{k,n+1} \in(\hat{r}_{k}, r_{k})$,}\\
        (j_{k+3,n}, j_{k+1, n+1}), &  \quad \text{if $j_{k+1,n+1}<r_{k+1}$ and $j_{k,n+1} > r_{k}$,}\\
        (j_{k+1,n+1}, j_{k+2,n+1}), &  \quad \text{if $j_{k+1,n+1}>r_{k+1}$.}\\
        \end{cases}
    \end{equation*}
\end{thm}
\begin{proof}
    The functions $G_{k,2}, G_{k,3}$ are continuous and do not intersect $F_k$ in the region $(0,j_{k+1,1})$ (by monotonicity in $k$ of $j_{k,1}$); hence, each function intersects $F_{k}$ exactly once in each region $(j_{k+1,n}, j_{k+1,n+1})$ for $n\in\mathbb{N}$. We note that the points $r=r_k, \hat{r}_k$ are the roots of $G_{k,3}, G_{k,4}$, respectively, and $r=r_{k+1}$ is the point of discontinuity of $G_{k,4}$. Hence, we have the following ordering
    \begin{equation*}
    \begin{cases}
        G_{k,2}(r) > G_{k,3}(r) > G_{k,4}(r) > 0, & r\in(0, \hat{r}_k),\\
        G_{k,2}(r) > G_{k,3}(r) > 0 > G_{k,4}(r), & r\in(\hat{r}_k, r_k),\\
        G_{k,2}(r) > 0 > G_{k,3}(r) > G_{k,4}(r), & r\in(r_k, r_{k+1}),\\
        G_{k,4}(r) > G_{k,2}(r) > 0 > G_{k,3}(r), & r > r_{k+1},\\
    \end{cases}
\end{equation*} 
    which directly yield the interlacing properties for the zeros of $\{J_k,J_{k+1}, J_{k+2}, J_{k+3}\}$, noting that $j_{k,n+1}<r_k$ is the condition for $G_{k,3}>0$ at the intersection point in the interval $(j_{k+1,n},j_{k+1,n+1})$.
    
    The function $G_{k,4}$ has a single discontinuity at $r = r_{k+1}$, with $G_{k,4}(r) < G_{k,3}(r)$ for all $r<r_{k+1}$ and $G_{k,4}(r) > G_{k,2}(r)$ for all $r>r_{k+1}$. If the discontinuity point $r= r_{k+1}$ lies between $(j_{k+1,n}, j_{k+1,n+1})$, then there is no intersection between $G_{k,4}$ and $F_k$, and otherwise there is exactly one intersection. The zeros satisfy the interlacing properties given by the inequalities for $G_{k,2}, G_{k,3}, G_{k,4}$, again using the fact that $F_k$ is a strictly decreasing function.
\end{proof}

\begin{figure}[ht!]
    \centering
    \includegraphics[width=0.8\linewidth]{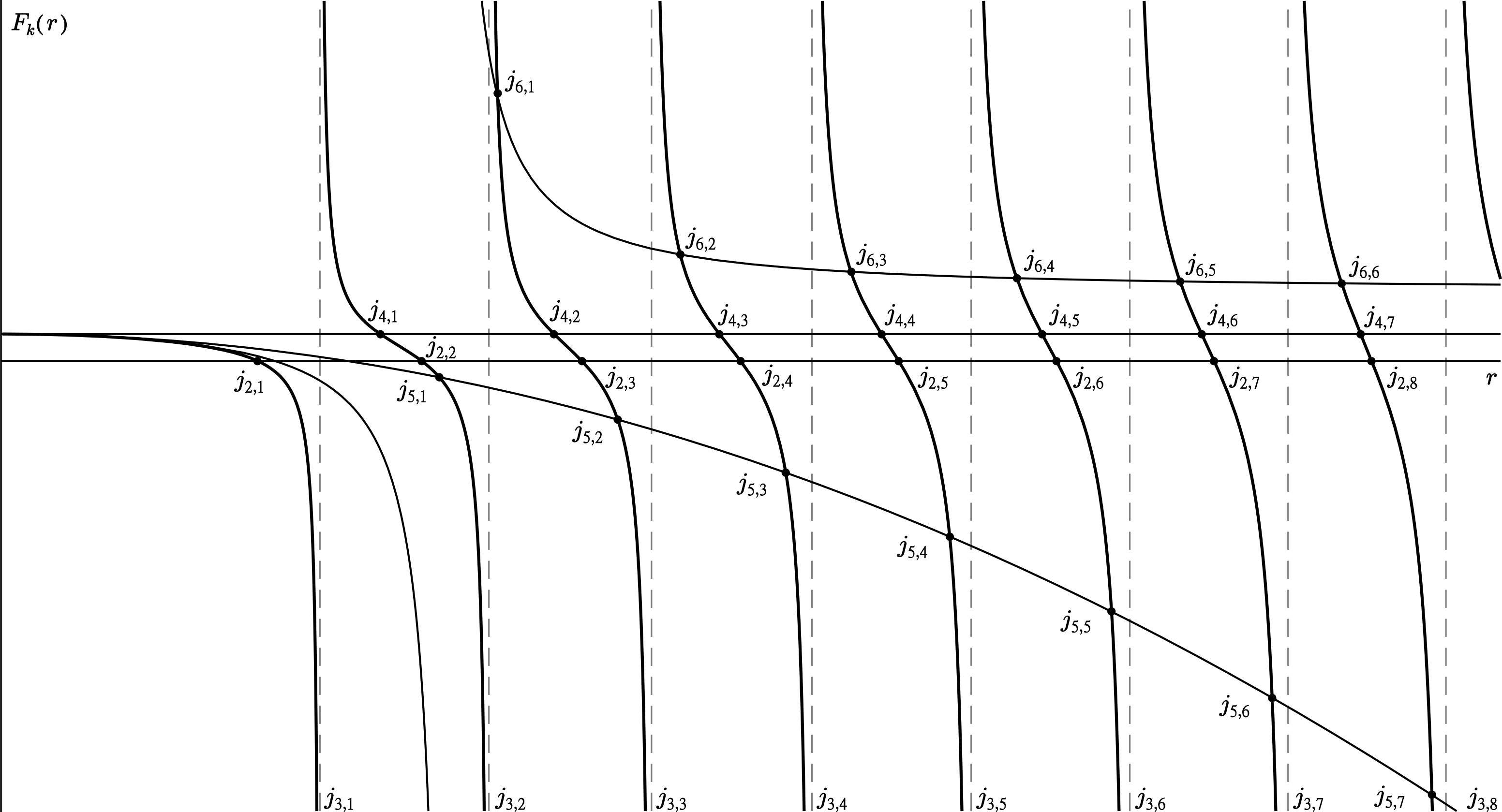}
    \caption{Numerical plot of the intersections between $F_k$ and $G_{k,2}, G_{k,3}, G_{k,4}$ for $k=2$, as well as the corresponding ordering for the zeros $j_{2,n}$, $j_{3,n}$, $j_{4,n}$, $j_{5,n}$ and $j_{6,n}$ for $n=1,\dots,8$.}
    \label{fig:m3}
\end{figure}

We present a specific example of Theorem~\ref{thm:interlacing} for $k=2$, see Figure~\ref{fig:m3}. In this case $r_{2} = 4\sqrt{3}$, $\hat{r}_2 = \sqrt{30}$, and $r_{3} = 4\sqrt{5}$, and so $r_{2},\hat{r}_2 < j_{2,2}$ and $r_{3}\in(j_{3,1}, j_{3,2})$. The first zero $j_{6,1}$ thus lies in the region $(j_{3,2}, j_{3,4})$, resulting in the interlacing property
\begin{equation*}
    j_{2,1} < j_{3,1} < j_{4,1} < j_{2,2}<j_{5,1}< \dots < j_{3,n+1} < j_{6,n} < j_{4,n+1} < j_{2,n+2} < j_{5,n+1} < j_{3,n+2} < \dots
\end{equation*}
for each $n\in\mathbb{N}$.

The approach presented herein provides a simple and intuitive way to prove the interlacing properties of zeros of Bessel functions. This approach is not only accessible---requiring only the identities \eqref{id:Bessel-1}, \eqref{id:Turan} and the analysis of curves of ratios of polynomials---but it also highlights additional intuition for how the interlacing between $j_{k,n}$ and $j_{k+\ell,n}$ varies for higher values of $\ell\in\mathbb{N}$. Via this method, one can calculate the $G_{k,\ell}$ whose intersection with $F_k$ generates $j_{k+\ell,n}$ and identify the discontinuities in $G_{k,\ell}$ that result in changes to interlacing of $j_{k+\ell,n}$ with $j_{k,n}$. We note that the approach relies heavily on the generic shape of the graph of $F_k$, which arises as a consequence of Sturm's oscillation theorem for Bessel functions. This suggests that a similar approach may be useful in characterising the zeros of other special functions, particularly those arising from Sturm--Liouville theory.

\bibliographystyle{abbrv}
\bibliography{Bibliography.bib}

\end{document}